\documentclass{amsart}
\usepackage{amsmath}
\usepackage{amsfonts}
\usepackage{amssymb}
\usepackage{graphicx}
\usepackage{epstopdf}
\usepackage{url}
\usepackage[section]{placeins}

\newtheorem{theorem}{Theorem}[section]

\newtheorem{proposition}[theorem]{Proposition}
\newtheorem{corollary}[theorem]{Corollary}

\begin{document}
\title{Some notes about the Zeta function and the Riemann hypothesis}

\author{Michele Fanelli}
\email[Michele Fanelli]{michele31fanelli@gmail.com}
\address [Michele Fanelli]{Via L.B. Alberti 5, 20149 Milano, Italy}
\author{Alberto Fanelli}
\email[Alberto Fanelli]{apuntoeffe@gmail.com}
\address[Alberto Fanelli] {Via L.B. Alberti 12, 20149 Milano, Italy}
\date{Mar 07, 2015}
\keywords{Riemann hypothesis, Gram, Backlund, Zeta function, extension, zeroes, number theory, 11M26 }

%
%
\begin{abstract}
The present essay aims at investigating whether and how far an algebraic analysis of the Zeta Function $Z(s)$ and of the RIEMANN Hypothesis can be carried out. Of course the well-established properties of the Zeta Function, explored in depth in over 150 years of world-wide study, are taken for granted.

The chosen approach starts from the recognized necessity of formulating an extension of the Zeta Function which is defined for $Re(s) = X < 1$.

The extension chosen substitutes a part of the RIEMANN series with an integral, see e.g. the work of GRAM and BACKLUND. The rest of the series is then replaced by two corrective terms.

The zero-condition on $Z(s)$ implies that a sum of three different types of terms should be nil. This zero-condition hypothetically satisfied in an outlying ($X\neq\frac{1}{2}$) complex-conjugate (c.c.) zero-pair implies the existence of another zero-pair, mirror-image of the first with respect to the Critical Line  $X=\frac{1}{2}$ (C.L.).

The term derived from the EULER-McLAURIN integration is endowed with a high degree of symmetry with respect to the C.L.. It follows that the other two terms of the zero-condition should possess the same type of symmetry, but this is shown to be possible only in the degenerate case when the two specular c.c. pairs collapse onto a single pair located on the C.L..
\end{abstract}

\maketitle
\tableofcontents

%
%
\section{Background}\label{sec:1}

The definition of the Zeta function,  $Z(s)$, i.e.  an analytical function of the complex variable  $s=X+i.Y$,  is 
\begin{equation}
Z(s)=\underset{n=1}{\overset{\infty }{\sum }}\frac{1}{n^s}=\underset{n=1}{\overset{\infty }{\sum }}\exp (-s.\ln n) \label{1}
\end{equation}

Such a function is defined only for   $Re (s)>1$,  but it is possible to find so-called ‘extensions’, i.e.  combinations of analytical functions which are defined over all of the complex plane and give the same values as the original definition over the half-plane    $X=Re (s)>1$. 

The Zeta function is connected with the distribution of prime numbers. For instance, the EULER formula for   $Z(s)$   is
\begin{equation}
Z(s)=\underset{P=2}{\overset{\infty }{\prod }}\frac{1}{1-P^{-s}} \label{2}
\end{equation}
where  $P$   stands for the generic prime.  

It was soon recognized that the locations of the complex-conjugate (c.c.) zeroes of   $Z(s)$ in the complex plane would have far-reaching implications for the search for large primes and their distribution among the natural numbers. 

In 1859 RIEMANN  \cite{Riem59} formulated his celebrated hypothesis to the effect that these zero-points were bound to exist only on the critical line $X=\frac{1}{2}$, and in 1900 HILBERT  \cite{Hilb02} included the problem of finding a proof (or a refutation) of this hypothesis among the 23 unsolved  challenges of the Mathematics of the coming century.
The XXth century saw the efforts of some of the best mathematicians being spent on the problem without success, and in the year 2000 the Clay Mathematics Institute updated the HILBERT statement by including the RIEMANN Hypothesis in a list of 7 still unsolved questions, which were indicated as the "problems of the coming Millennium"  \cite{Devl02}.

After the development of security codings  (RSA codes, see \cite{Rive77})  based on the search for very large primes, it was seen that a proof, or a refutation, of the RIEMANN Hypothesis would entail  important practical consequences.

Intense studies on the Zeta function carried out in the XXth century  led to the definition of three categories of zeroes:
\begin{itemize}
  \item The ‘trivial’ zeroes, located on the real axis   $Y=0$   at abscissas    $X=-2, -4, -6, \dots, -2.k, \dots$  where $k \in \mathbb{N},$ $k\geq1$.  These zeroes, the only ones to belong to the real axis, are devoid of interest in the context of the present analysis;
  \item The ‘canonical’ zeroes, located on the critical line. HARDY \cite{Hard14} showed that there are infinitely many such canonical zeroes, which occur at pairs of complex conjugate values of  $s$, i.e. at  $s=s_0=X_0+i.Y_0$  and at  $s=\overline{s}_0=X_0-i.Y_0$, with   $Y_0\neq 0$  and   $X_0=\frac{1}{2}$. Many billions of such zeroes have been calculated, for higher and higher values of  $Y_0$, see ODLYZKO \cite{Odly89} \cite{Odly01}, all belonging to the critical line (C.L.)   $X_0=\frac{1}{2}$; 
  \item The hypothetical ‘outlying’ zeroes. Those should appear within the critical strip   $0<Re (s)<1$   as two pairs of complex conjugate values of  $s$, specularly disposed with respect to the C.L., i.e. at   $X_0=\frac{1}{2}\pm \xi $, with   $-\frac{1}{2}<\xi <\frac{1}{2}$, $\xi \neq 0$. No example of an outlying zero has ever been found to-date.
\end{itemize}

The possibility of occurrence of the canonical and of the outlying zeroes (if any) was shown  (see EDWARDS \cite{Edwa74}) to be restricted to the ‘critical strip’  $0<Re (s)<1$. The search for outlying zeroes was carried out by numerical exploration in the critical strip, up to extremely high values of  $Y$, without finding any outlier.  Thus the certainty was gained that outlying zeroes could occur only extremely far up (and far down) in the critical strip. But no finite numerical search can ever be exhaustive, so that the search for a theoretical proof is the only viable alternative.

However, to date a theoretical proof that such outlying zeroes cannot exist has not been found, nor has it been possible to find a theoretical proof of their real existence.

%
%
\section{Extensions of the Zeta function}\label{sec:2}
As already said, the original definition of the Zeta function, see  Eq.\eqref{1}, is not adequate to investigate the distribution of the zeroes, insofar as it yields definite values for  $Z(s)$   only if  $Re (Z)=X>1$, i.e. in a region of the complex plane which does not contain any zero.  It is therefore necessary to work on extensions of the Zeta function valid over all of the complex plane.  These extensions are obtained by analytical continuation.

There is not a unique opportunity to formulate such an extension, so the Authors chose to work with extensions that can be represented by the following general form:
\begin{equation}
Z(s)=A.\left[\frac{1}{1-s}+s.Q(s)\right]\label{3}
\end{equation}
where  $A$  is a finite constant,   $A \neq 0$, and   $Q(s)$   is an analytical function of   $s$. Such a form derives from the replacement of part of  the sum
\begin{equation*}
\underset{n=1}{\overset{\infty }{\sum }}\frac{1}{n^s} 
\end{equation*}
with an integral plus a sum of corrective terms\footnote{Note that $\int x^{-s}.dx=\frac{x^{1-s}}{1-s}$. For the corrective terms $\sum CT(s)$ see further on  (footnote 2 and APPENDIX).}: 
\begin{equation}
\underset{n=1}{\overset{\infty }{\sum }}\frac{1}{n^s}=\underset{n=1}{\overset{N}{\sum }}\frac{1}{n^s}+\int _{N+\frac{1}{2}}^{\infty }x^{-s}.dx+\sum CT(s,N)\label{4}
\end{equation}

This extension, obtained by an EULER-Mc LAURIN integration technique, is not an approximation, but an exact equivalence. Such an extension was developed, e.g., by GRAM \cite{Gram03} and BACKLUND \cite{Back18}. Also the Authors of the present note developed an extension of the same general form, see FANELLI \& FANELLI  \cite{Fane10}. Performing legitimate algebraic operations on Eq.\eqref{3} the zero condition can be put under the form:
\begin{equation}
\frac{1}{s.(1-s)}+Q(s)=0\label{5}
\end{equation}
or
\begin{equation}
s.(s-1)=\frac{1}{Q(s)}\label{6}
\end{equation}

Let us now work on Eq.\eqref{6}. Putting for simplicity   $-\frac{1}{Q(s)}=q(s)$ the \underline{zero condition} is \footnote{The formulation  from  which the following developments take their start is:
$Z(s)=\left[\frac{x_0}{s-1}-\underset{k=0}{\overset{\infty }{\sum }}\beta _{2.k+1}.\frac{\Gamma(s+2.k+1)}{\Gamma(s)}+\sum _{n=1}^{n=N}\left(\frac{n}{x_0}\right)^{-s}\right].\exp\left(-s.\ln{x_0}\right)$
where the coefficients  $\beta _{2.k+1}$  are numerical constants of alternate sign (see APPENDIX). In the quest for zeroes the non-null factor    $\exp (-s.\ln x_0)$  can be dropped.}:
\begin{equation}
s^2-s+q(s)=0\label{7}
\end{equation}

Let us find next  what are the necessary implications of the assumption that an outlying zero has been found at location:
\begin{equation}
s=s_1=\frac{1}{2}+\xi +i.Y\label{8}
\end{equation}
with  $\xi \neq 0$.  Eq.\eqref{7} will of course hold also in the point  which is the c.c. of \eqref{8}, i.e. at:
\begin{equation}
s=\overline{s}_1=\frac{1}{2}+\xi -i.Y\label{9}
\end{equation}
with  $\xi \neq 0$.  Eqs. \eqref{8} and \eqref{9} imply that the two values of $s$  assumedly  fulfilling the zero-condition should comply with the following equations:
\begin{equation}
\left(s-s_1\right).\left(s-\overline{s}_1\right)=0 \text{ from which } s^2-\left(s_1+\overline{s}_1\right).s+s_1.\overline{s}_1=0\label{10}
\end{equation}
 or, after Eqs. \eqref{8} and \eqref{9}:
\begin{equation*}
s^2-(1+2.\xi ).s+\frac{1}{4}+\xi ^2+\xi +Y^2=0
\end{equation*}
 to be compared to Eq.\eqref{7}.   Needless to say, the last equation yields back solutions  \eqref{8} and \eqref{9}.
The term  $q(s)=s.(1-s)$   in the zero-condition \eqref{7} is written down, see footnote 2 and APPENDIX, as: 
\begin{equation}
q(s)=\frac{-s.\Gamma(s)}{\frac{1}{x_0}.\underset{k=0}{\overset{\infty }{\sum }}\beta _{2.k+1}.\Gamma(s+2.k+1)-\frac{\Gamma(s)}{x_0}.\sum _{n=1}^{n=N}\left(\frac{n}{x_0}\right)^{-s}}\label{11}
\end{equation}

%
%
\section{The Hadamard-De La Vallee-Poussin theorem for the hypothetical outlying zeroes}\label{sec:3}

HADAMARD  \cite{Hada96} and DE LA VALLEE-POUSSIN  \cite{Dela96} showed that the zeroes \eqref{9} with $\xi \neq 0$ should occur as quadruplets, i.e. two c.c. pairs specularly  disposed with respect to the C.L.:
\begin{equation}
s=\frac{1}{2}\pm \xi \pm i.Y\label{12}
\end{equation}

In other words, if  $Z(s)=0$ in  the first c.c. pair, it is again  $Z(s)=0$   interchanging    $s$   with    $1-s$;    see Fig. ~\ref{fig:1} for greater clarity. This is referred to in the following as ‘the H-DLVP theorem’.
Let us now elaborate on the necessary  implications of the zero condition \eqref{7} and of the H-DLVP findings, under the assumption that a quadruplet,  Eq.\eqref{12},  of outlying zeroes has been found.
Let us remark, first of all, that in any of the assumed outliers   $s$   and   $q(s)$   take on definite complex values. And let us define, for an orderly treatment, the two uppermost members of the quadruplet of assumed outliers:
\begin{equation}
s_1=\frac{1}{2}+\xi +i.Y_0  \text{,  }   s_2=\frac{1}{2}-\xi +i.Y_0\label{13}
\end{equation}
(the first to the right of the C.L., the second one to the left, see Figure ~\ref{fig:1}). 

The two lowermost members of the quadruplet of assumed outliers are then defined by:
\begin{equation*}
s_3=\frac{1}{2}+\xi -i.Y_0=1-s_2=\overline{s_1}    \text{,  }   s_4=\frac{1}{2}-\xi -i.Y_0=1-s_1=\overline{s_2}
\end{equation*}
(in the same order).

By trivial algebra it is readily seen that the following identities necessarily hold: 
\begin{equation}
\begin{split}
1-s_1=s_4 &\qquad 1-s_2=s_3 \\
s_1.\left(1-s_1\right)=s_1.\overline{s_2} &\qquad s_3.\left(1-s_3\right)=s_2.\overline{s_1} \\
s_2.\left(1-s_2\right)=s_2.\overline{s_1}=\overline{s_1.\left(1-s_1\right)}   &\qquad  s_4.\left(1-s_4\right)=s_1.\overline{s_2}=\overline{s_3.\left(1-s_3\right)} 
\end{split}\label{14}
\end{equation}

Thus in the assumed zero-point  1  the sum of the first two terms of the L.H.S. of Eq.\eqref{7}  is the c.c. of the sum of the first two terms of the L.H.S. of the analogous equation evaluated in the specular zero-point 2, and  the sum of the first two terms of the L.H.S. of Eq.\eqref{7}  evaluated in the assumed zero-point 3 is the c.c. of the sum of the first two terms of the L.H.S. of the analogous equation evaluated in the specular zero-point 4. Therefore, for the hypothetical quadruplet of outliers to exist, the same equality should necessarily hold also for the  relevant  third term:

\begin{equation}
\begin{split}
&q\left(s_1\right)=\overline{q\left(s_2\right)}=\overline{q\left(1-\overline{s_1}\right)}=q\left(1-s_1\right)       \text{ and }\\
&q\left(s_3\right)=\overline{q\left(s_4\right)}=\overline{q\left(1-s_1\right)}=q\left(1-\overline{s_1}\right)=q\left(s_2\right)=q\left(1-s_3\right)\\
\\
&q\left(s_1\right)=q\left(\frac{1}{2}+\xi +i.Y_0\right)=q\left(\frac{1}{2}-\xi -i.Y_0\right)      \text{ and }\\
&q\left(s_3\right)=q\left(\frac{1}{2}+\xi -i.Y_0\right)=q\left(\frac{1}{2}-\xi +i.Y_0\right)=\overline{q\left(s_1\right)} 
\end{split}\label{15}
\end{equation}

According to Eqs.\eqref{7}, \eqref{11}, the following equality should necessarily hold in any hypothetical outlying zero:
\begin{equation}
\begin{split}
&q(s)=\frac{-s.\Gamma(s)}{\frac{1}{x_0}.\underset{k=0}{\overset{\infty }{\sum }}\beta _{2.k+1}.\Gamma(s+2.k+1)-\frac{\Gamma(s)}{x_0}.\sum _{n=1}^{n=N}{\left(\frac{n}{x_0}\right)}^{-s}}\\
&=s.(1-s)=\frac{1}{4}-\xi ^2+Y_0^2-i.2.\xi .Y_0\label{16}
\end{split}
\end{equation}

Note that if    $\xi=0$,  i.e. on the C.L., then $\Gamma (s)=\sqrt{\frac{\pi }{Ch\left(\pi .Y_0\right)}}=$ real and $q(s)=\frac{1}{4}+Y_0^2$.

Therefore the Riemann Hypothesis could be proved by showing that  Eq.\eqref{16} cannot be satisfied for  $\xi \neq 0$   (but of course for $s$ inside the critical strip, i.e. for   $0<|\xi |<\frac{1}{2}$).

If we could assume that $\frac{1}{Q(s)}=-q(s)$   is real, then a first part of a proof path could be summarized as follows:

\begin{equation}
\begin{split}
&s_1.s_4=s_1.\overline{s_2}=\overline{s_2.\overline{s_1}} \text{= real ; then by }  \eqref{16}\\
&\xi .Y_0=0 \\
&\xi=0\\
&q\left(s_1\right)=s_1.\left(1-s_1\right)=q\left(s_2\right)=s_2.\left(1-s_2\right)  \text{= real ;}\\
&\text{therefore, recalling Eqs.} \eqref{7} \text{ and } \eqref{11} \\
&q(s_1)=q(s_2)\\
&= \left|\frac{-s_1.\Gamma\left(s_1\right)}{\frac{1}{x_0}.\underset{k=0}{\overset{\infty }{\sum }}\beta _{2.k+1}.\Gamma\left(s_1+2.k+1\right)-\frac{\Gamma\left(s_1\right)}{x_0}.\sum _{n=1}^{n=N}{\left(\frac{n}{x_0}\right)}^{-s_1}}\right|\\
&=\frac{1}{4}+Y_0^2\label{17}
\end{split}
\end{equation}
where    $ s_1=\frac{1}{2}+i.Y_0$. The above formula can be used for analyzing the symmetry features of the different components of the canonical zeroes of Zeta; see APPENDIX  [see also Eqs.  \eqref{6} to \eqref{11}].

It is seen that this chain of inferences depends very closely on the form of the EULER-Mc LAURIN integral:

\begin{equation*}
\int x^{-s}.dx=\frac{x^{1-s}}{1-s}
\end{equation*}
(see footnote 1) as well as on the symmetries of the product   $s.(1-s)$   with respect to the C.L. vs. the symmetries, or the lack of them, of the function:
\begin{equation*}
q(s)=\frac{-s.\Gamma(s)}{\frac{1}{x_0}.\underset{k=0}{\overset{\infty }{\sum }}\beta _{2.k+1}.\Gamma(s+2.k+1)-\frac{\Gamma(s)}{x_0}.\sum _{n=1}^{n=N}{\left(\frac{n}{x_0}\right)}^{-s}}
\end{equation*}

Let us now tackle the more general assumption that  $\frac{1}{Q(s)}$    is complex. What follows is based on the extension of the Zeta function from which Eq.\eqref{11} is derived, but a similar chain of inferences could be followed by adopting other extensions based as well on a EULER-Mc LAURIN type of integration, such as, e.g., the GRAM-BACKLUND extension (see \cite{Gram03} and \cite{Back18}). Therefore, apart from an inessential non-zero factor,  the starting point of the proposed argumentation is the analysis of the necessary constraints to be imposed for the existence of the quadruplet of outliers, represented, see Eq.\eqref{3}, by the equation:

\begin{equation}
\frac{Z}{A.s}=\frac{1}{s.(1-s)}+Q(s)=0 \text{ or } s^2-s+q(s)=0 \text{, with  }   q(s)=-\frac{1}{Q(s)}  \text{,  } \label{18}
\end{equation}
where all terms are  analytical functions of   $s $.  

In what follows it is important to keep in mind the findings of HADAMARD \cite{Hada96} and DE LA VALLEE-POUSSIN  \cite{Dela96} (henceforth referred to, for brevity, as H-DLVP):  the condition \eqref{18} should  be fulfilled in every zero-point, e.g. in     $s=s_1=\frac{1}{2}+\xi +i.Y_0$,  as well as in the specular zero-point of the H-DLVP quadruplet     $s=s_2=\frac{1}{2}-\xi +i.Y_0$.

\begin{proposition}
In the four outliers foreseen by H-DLVP, i.e. for   $s=\frac{1}{2}\pm \xi \pm i.Y_0$, the  zero-conditions on  $Z(s)$   and the relationships between  $ s_1   ,s_2   ,s_3,   s_4 $  imply the necessary consequence  $\xi=0$. 
\end{proposition}

\begin{proof}
The zero-condition:
\begin{equation}
\left(s-\frac{1}{2}\right)^2-\frac{1}{4}+q(s)=0 \label{19}
\end{equation}
must be fulfilled in the four hypothetical outliers   1, 2, 3, 4 (see Fig. ~\ref{fig:1}), i.e. in   $s=\frac{1}{2}\pm \rho .\exp \left[i.\left(\frac{\pi }{2}\pm \alpha \right)\right]$    with 
\begin{equation}
\begin{split}
\alpha &\neq 0  \text{,  } q(s)=\frac{1}{G(s,K)+F(s,N)} \text{ and  } \\
G(s,K)&=\frac{\frac{1}{x_0}.\underset{k=0}{\overset{K}{\sum }}\beta _{2.k+1}.\Gamma (s+2.k+1)}{-s.\Gamma (s)}  \text{,  } F(s,N)=\frac{\underset{n=1}{\overset{n=N}{\sum }}\left(\frac{n}{x_0}\right)^{-s}}{x_0.s}\label{20}
\end{split}
\end{equation}

The H-DLVP necessary conditions to be complied with at points  1, 2, 3, 4  are:   $q(s_2 )=q(s_3 )=\overline{q}(s_1 )$, i.e.:
\begin{equation}
\begin{split}
q\left(\frac{1}{2}+\rho .\exp \left[i.\left(\frac{\pi }{2}+\alpha \right)\right]\right)&\\
=q\left(\frac{1}{2}-\rho .\exp \left[i.\left(\frac{\pi }{2}+\alpha \right)\right]\right)&\\
=\overline{q}\left(\frac{1}{2}+\rho .\exp \left[i.\left(\frac{\pi }{2}-\alpha \right)\right]\right)\label{21}
\end{split}
\end{equation}
where   $\rho .\exp \left[i.\left(\frac{\pi }{2}-\alpha \right)\right]=\xi +i.Y_0$.  Conditions \eqref{21} are  identically  fulfilled by the function $\left(s-\frac{1}{2}\right)^2-\frac{1}{4}$, thus it follows that they should be fulfilled  by the function  $q(s)=\frac{1}{G(s,K)+F(s,N)}$  as well, i.e. it should be:
\begin{equation}
G\left(s_2,K\right)+F\left(s_2,N\right)=G\left(s_3,K\right)+F\left(s_3,N\right)=\overline{G}\left(s_1,K\right)+\overline{F}\left(s_1,N\right) \label{22}
\end{equation}

Since  the function  $G(s,K)$  is analytical and  $s_3=\overline{s_1}$  it is $G\left(s_3,K\right)\equiv \overline{G}\left(s_1,K\right)$, so that in place of  Eq.\eqref{22} the single condition to be fulfilled is: 
\begin{equation}
\begin{split}
G\left(s_2,K\right)+F\left(s_2,N\right)&=G\left(s_3,K\right)+F\left(s_3,N\right)\\
=\frac{-1}{\left(s_2-\frac{1}{2}\right)^2-\frac{1}{4}}&=\frac{-1}{\left(s_3-\frac{1}{2}\right)^2-\frac{1}{4}} \label{23}
\end{split}
\end{equation}

The function  $\frac{1}{\left(s-\frac{1}{2}\right)^2-\frac{1}{4}}$ is characterized by the double symmetry  $\frac{1}{\left(s_2-\frac{1}{2}\right)^2-\frac{1}{4}}=\frac{1}{\left(s_3-\frac{1}{2}\right)^2-\frac{1}{4}}=\frac{1}{\left(\overline{s_1}-\frac{1}{2}\right)^2-\frac{1}{4}}$, which holds for  $\alpha \neq 0$   (in the following we assume $\alpha >0$, in agreement with Fig. ~\ref{fig:1}). 

If this symmetry is shared by $ F(s,N)$ it can be shown that conditions \eqref{23},  with  $\alpha \neq 0$,  require that it be shared also by function  $ G(s,N)$. If  this does not happen Eqs. \eqref{23} are falsified and the H-DLVP quadruplets of outliers cannot exist.

Let us begin by checking the consistency of these constraints for function $F(s,N)$, where the factor $\frac{1}{x_0}$  is inessential as long as the symmetry properties are to be investigated: thus we are to characterize the symmetry of function
\begin{equation}
x_0.F(s,N)=\frac{\underset{n=1}{\overset{n=N}{\sum }}\left(\frac{n}{x_0}\right)^{-s}}{s}  \text{,  } n<x_0 \label{24}
\end{equation}

Since the upper limit of the sum,   $n=N$, is to some extent arbitrary, the symmetry condition: 
\begin{equation}
\frac{\underset{n=1}{\overset{n=N}{\sum }}\left(\frac{n}{x_0}\right)^{-s}}{s}-\frac{\underset{n=1}{\overset{n=N}{\sum }}\left(\frac{n}{x_0}\right)^{s-1}}{1-s}=0 \label{25}
\end{equation}
is to be enforced for the generic term of the sum:
\begin{equation}
\frac{\left(\frac{n}{x_0}\right)^{-s}}{s}-\frac{\left(\frac{n}{x_0}\right)^{s-1}}{1-s}=0 \label{26}
\end{equation}

Posing successively: 	
\begin{equation}
\begin{split}
s&=\frac{1}{2}+\rho .\exp \left[i.\left(\frac{\pi }{2}-\alpha \right)\right]=\frac{1}{2}+W   \text{,  } \alpha >0 \\
s-1&=-\frac{1}{2}+\rho .\exp \left[i.\left(\frac{\pi }{2}-\alpha \right)\right]=-\frac{1}{2}+W\\
W&=U+i.V=\rho .(\sin \alpha +i.\cos \alpha ) \label{27}
\end{split}
\end{equation}
we get by trivial algebra:	
\begin{equation}
\begin{split}
\frac{\left(\frac{n}{x_0}\right)^{-s}}{s}-\frac{\left(\frac{n}{x_0}\right)^{s-1}}{1-s}&=\frac{\exp \left[-s.\ln \left(\frac{n}{x_0}\right)\right]}{s}+\frac{\exp \left[(s-1).\ln \left(\frac{n}{x_0}\right)\right]}{s-1}\\
&=\sqrt{\frac{x_0}{n}}.\frac{2.W.\text{Ch}\left[W.\ln \left(\frac{n}{x_0}\right)\right]+\text{Sh}\left[W.\ln \left(\frac{n}{x_0}\right)\right]}{W^2-\frac{1}{4}}=0 \label{28}
\end{split}
\end{equation}

And recalling that: 
\begin{equation}
\begin{split}
U&=\rho .\sin \alpha >0 \\
V&=\rho .\cos \alpha >0 \label{29}
\end{split}
\end{equation}
with elementary manipulations we get: 
\begin{equation}
\begin{split}
U.Ch^2\left[U.\ln \left(\frac{n}{x_0}\right)\right].\sin \left[2.V.\ln \left(\frac{n}{x_0}\right)\right]-\\
V.Sh\left[2.U.\ln \left(\frac{n}{x_0}\right)\right].\sin ^2\left[V.\ln \left(\frac{n}{x_0}\right)\right]+\\
\frac{1}{4}.Sh\left[2.U.\ln \left(\frac{n}{x_0}\right)\right].\sin \left[2.V.\ln \left(\frac{n}{x_0}\right)\right]&=0 \\
\\
V.Sh\left[2.U.\ln \left(\frac{n}{x_0}\right)\right].\cos ^2\left[V.\ln \left(\frac{n}{x_0}\right)\right]+\\
U.Sh^2\left[U.\ln \left(\frac{n}{x_0}\right)\right].\sin \left[2.V.\ln \left(\frac{n}{x_0}\right)\right]+\\
\frac{1}{4}.Sh\left[2.U.\ln \left(\frac{n}{x_0}\right)\right].\sin \left[2.V.\ln \left(\frac{n}{x_0}\right)\right]&=0\\
\\
\left[
\begin{array}{cc}
 1 & -Sh\left[2.U.\ln \left(\frac{n}{x_0}\right)\right] \\
 -1 & Sh\left[2.U.\ln \left(\frac{n}{x_0}\right)\right]
\end{array}
\right].\left\{
\begin{array}{c}
 U \\
 V
\end{array}
\right\}&=\left\{
\begin{array}{c}
 0 \\
 0
\end{array}
\right\}\\
\\
 \text{ with  } Det\left[
\begin{array}{cc}
 1 & -Sh\left[2.U.\ln \left(\frac{n}{x_0}\right)\right] \\
 -1 & Sh\left[2.U.\ln \left(\frac{n}{x_0}\right)\right]
\end{array}
\right]&=0 \\
\text{so that non-null solutions for }  U \text{ and }  &V  \text{ are possible. } \label{30}
\end{split}
\end{equation}

In Eq.\eqref{30} it is $\ln \left(\frac{n}{x_0}\right)<0$   because  $n<x_0$. Therefore if it were  $U\neq0$    it would be $ V<0$  and $V$ dependent from $n$,  contrary to assumptions made. Therefore we get the necessary constraint:
\begin{equation}
\begin{split}
&U=0\rightarrow sin\alpha =0\rightarrow cos\alpha =1  \text{, } W=i\rho \in \mathbb{I}\\
&\rho \neq 0 \text{, } indeterminate \label{31}
\end{split}
\end{equation}
so that another condition is needed in order to determine the c.c. pair in which   $Z=0$  [see Section \ref{sec:4}., Eq.\eqref{38}]. 

Let us at last take into account the third function:	
\begin{equation}
\begin{split}
G(s,K)&=\frac{\frac{1}{x_0}.\underset{k=0}{\overset{K}{\sum }}\beta _{2.k+1}.\Gamma (s+2.k+1)}{-s.\Gamma (s)}\\
&=-\frac{1}{x_0}.\underset{k=0}{\overset{K}{\sum }}\beta _{2.k+1}.(s+1).(s+2)\ldots (s+2.k) \label{32}
\end{split}
\end{equation}

where the external factor $-\frac{1}{x_0}$, inessential as long as the symmetry properties are to be investigated, can be dropped, so that the condition to be enforced is: 
\begin{equation}
\begin{split}
&\underset{k=0}{\overset{K}{\sum }}\beta _{2.k+1}.(s+1).(s+2)\ldots (s+2.k)\\
=&\underset{k=0}{\overset{K}{\sum }}\beta _{2.k+1}.(1-s+1).(1-s+2)\ldots (1-s+2.k) \label{33}
\end{split}
\end{equation}

But also for this function the symmetry condition \eqref{33}  should not depend from the upper limit $K$  of the summation operator (which affects also the real constants $\beta _{2.k+1}$ ), so that finally the condition to be respected is:
\begin{equation}
[(s+1).(s+2)\ldots (s+2.k)-(-s+2).(-s+3)\ldots (-s+2.k+1)]=0 \label{34}
\end{equation}
which transforms successively into:
\begin{equation}
[(s+1).(s+2)\ldots (s+2.k)-(s-2).(s-3)\ldots (s-2.k-1)]=0 \label{35}
\end{equation}
and after  \eqref{27}: 
\begin{equation}
\begin{split}
\left[\rho .(\sin \alpha +i.\cos \alpha )+\frac{3}{2}\right].\left[\rho .(\sin \alpha +i.\cos \alpha )+\frac{5}{2}\right]&\ldots\\
\left[\rho .(\sin \alpha +i.\cos \alpha )+2.k+\frac{1}{2}\right]&-\\
\left[\rho .(\sin \alpha +i.\cos \alpha )-\frac{3}{2}\right].\left[\rho .(\sin \alpha +i.\cos \alpha )-\frac{5}{2}\right]&\ldots\\
\left[\rho .(\sin \alpha +i.\cos \alpha )-2.k-\frac{1}{2}\right]&=0 \label{36}
\end{split}
\end{equation}

Since from the preceding analysis it has been found that  $\sin⁡ \alpha =0$,   $\cos⁡ \alpha =1$, Eq.\eqref{36} gives:
\begin{equation}
\begin{split}
&\left[i.\rho +\frac{3}{2}\right].\left[i.\rho +\frac{5}{2}\right]\ldots \left[i.\rho +2.k+\frac{1}{2}\right]-\\
&\left[i.\rho -\frac{3}{2}\right].\left[i.\rho -\frac{5}{2}\right]\ldots \left[i.\rho -2.k-\frac{1}{2}\right]=0\label{37}
\end{split}
\end{equation}
and it can be shown that the L. H. S. of Eq.\eqref{37} is reduced to a purely imaginary expression, which is the double of the imaginary component of the L. H. S. of Eq.\eqref{33}.
\end{proof}

\begin{corollary}
Therefore the symmetry conditions connected to the H-DLVP theorem imply that the imaginary components of  $Q(s)$ originating from  all of the three different types of terms  give a null sum, so that in the end in each zero-point  $Q(s)$  is a purely real quantity.

Summing up, it has been tentatively proved that
\begin{itemize}
\item the assumed outlying zero-point cannot exist   ($\sin⁡ \alpha =0$)  
\item the function  $Q(s)$  takes on  a real value in the actual zero-points, which all lie on the C.L.
\end{itemize}

Thus the said conditions constrain the four hypothetical H-DLVP zero-points to collapse onto a single pair of c.c. zero-points lying on the C.L. 

This conclusion, which seems to lend a prima facie support to the RIEMANN Hypothesis,  is  tied to the different symmetry features of the quadratic form  $\left(s-\frac{1}{2}\right)^2-\frac{1}{4}$   vs. those of the functions appearing in the expression of   $q(s)$.

\end{corollary}

%
%
\section{Further comments on the necessary implications of the zero-condition and of the H-DLVP theorem for an hypothetical outlying zero. The explicit formula for $Y_0$  }\label{sec:4}

Eqs. from  \eqref{19} to \eqref{30} are the consequence of necessary conditions, which constrain the c.c. zero-points of function Zeta to lie on the C.L., i.e. they yield  $X_0=\frac{1}{2}$   for any c.c. zero-pair. They are, however, by no means sufficient to determine the actual position of those zero-points, i.e. the actual values of   $Y_0$,   see Section \ref{sec:1}. In order to obtain these values, it is necessary to use the explicit formulation, imposing (see APPENDIX):

\begin{equation}
\begin{split}
\frac{1}{4}+Y_0^2&=s.(1-s)\\
&=\frac{-s.\Gamma (s)}{\underset{k=0}{\overset{K}{\sum }}x_0^{-2.(k+1)}.\alpha _{2.k+1}.\Gamma (s+2.k+1)-\Gamma (s).\underset{n=1}{\overset{n=N}{\sum }}\frac{n^{-s}}{x_0^{1-s}}}\label{38}
\end{split}
\end{equation}
where  $s=\frac{1}{2}+i.Y_0$. Eq.\eqref{23} is an implicit transcendent equation where the only unknown is   $Y_0$; it admits of infinitely many solutions, yielding thus the positions of the canonical zeroes of the Zeta function along the C.L.

%
%
\section{Conclusions}\label{sec:5}

The developments and considerations carried out in the preceding Sections of the present Note seem to lead to the conclusion that the only non-trivial, c.c. zeroes of the Zeta Function are necessarily located along the critical line $X=X_0=\frac{1}{2}$, thus bringing preliminary evidence in favor of a verification of the RIEMANN Hypothesis, on the basis of only algebraic operations. It would be of course necessary that such a conclusion be corroborated by a more rigorous formal analysis, which is outside the field of competence of the Authors.

The Authors pose themselves a quite  obvious question, namely: has such an approach ever been attempted? If never so, why? Or else, in the former studies of the problem, has it been considered and rejected? The Authors could not find any published answer, nor any hint one way or the other; however, admittedly their knowledge of the relevant specialized literature is far from exhaustive.

Many a professional mathematician, privately consulted, could not find any manifest fault in the Authors’ treatment, but expressed the opinion that quite probably a rigorous, ‘fool-proof’ verification (or falsification) of the RIEMANN Hypothesis cannot be achieved by ‘elementary’ methods and that therefore it is useless even to try. Thus the Authors feel that  a final verdict cannot yet be cast, and are led to consider that their efforts, summarily sketched in the preceding Sections, though prima facie suggestive, leave the question open to further debate.

%
%
\section{Appendix}\label{sec:6}
DETAILS OF THE EULER-McLAURIN INTEGRATION  AND STUDY OF THE CONVERGENCE OF THE FUNCTIONS ADOPTED FOR THE SUM REPRESENTING THE ZETA FUNCTION

As said in the text of the present essay, some of the extensions developed to work on the properties of the Zeta  function are based on a EULER-McLAURIN type of integration plus some corrective terms: 

\begin{equation}
Z(s)=\underset{n=1}{\overset{n=N}{\sum }}n^{-s}+\underset{x_0=N+\frac{1}{2}}{\overset{\infty }{\int }}x^{-s}.dx+\underset{k}{\sum }CT_k\left(s,k,x_0\right)\label{A1} \tag{A1}
\end{equation}

Such are, e.g., the GRAM-BACKLUND extension, see  \cite{Gram03} and \cite{Back18},  as well as the simpler form adopted by the Authors of the present essay.

Some delicate questions need, however, to be discussed: these concern on one hand the integral, which could turn up to be an improper integral on the account of the infinite upper limit whereby the integral function $\int x^{-s}.dx$ may diverge to infinity, and on the other hand the convergence to a definite value of the sum of the corrective terms  $\underset{k}{\sum }CT_k\left(s,k,x_0\right)$.

This APPENDIX is devoted to investigating these important questions. 

Let us begin by recalling the formal structure of the analytical function we are working with: 	
\begin{equation}
Z(s)=\underset{n=1}{\overset{\infty }{\sum }}n^{-s}=\underset{n=1}{\overset{\infty }{\sum }}e^{-s.\ln n} \label{A2} \tag{A2}
\end{equation}
where :
\begin{equation}
\begin{split}
s=\frac{1}{2}+\xi +i.Y_0\\
-\frac{1}{2}<\xi <+\frac{1}{2}
\end{split} \label{A3} \tag{A3}
\end{equation}

Developing Eq.\eqref{A1} under our assumptions one finds (see further on in this APPENDIX):
\begin{equation}
\begin{split}
Z(s)=&\underset{n=1}{\overset{N}{\sum }}n^{-s}+\underset{X\rightarrow \infty }{\lim }\frac{X^{1-s}-\left(N+\frac{1}{2}\right)^{1-s}}{1-s}+\\
&\underset{K,N\rightarrow \infty }{\lim }\left[\frac{\left(N+\frac{1}{2}\right)^{-s-1}}{\Gamma (s)}.\underset{k=0}{\overset{k=K}{\sum }}\alpha _{2.k+1}.\Gamma (s+2.k+1).\left(N+\frac{1}{2}\right)^{-2.k}\right]
\end{split} \label{A4} \tag{A4}
\end{equation}

Indeed, from the definition of the Zeta function:
\begin{equation}
Z(s)=\underset{n=1}{\overset{\infty }{\sum }}n^{-s}=1^{-s}+2^{-s}+3^{-s}+\ldots \label{A5} \tag{A5}
\end{equation}
it is possible to operate as follows.

Starting from the identity:
\begin{equation}
\begin{split}
&1^{-s}+2^{-s}+3^{-s}+\ldots = \\
&\underset{n=1}{\overset{n=N}{\sum }}n^{-s}+\int _{N+\frac{1}{2}}^{N+\frac{3}{2}}x^{-s}.dx+\\
&\int _{N+\frac{1}{2}}^{N+\frac{3}{2}}\left[(N+1)^{-s}-x^{-s}\right].dx+\int _{N+\frac{3}{2}}^{N+\frac{5}{2}}x^{-s}.dx+\\
&\int _{N+\frac{3}{2}}^{N+\frac{5}{2}}\left[(N+2)^{-s}-x^{-s}\right].dx+\int _{N+\frac{5}{2}}^{N+\frac{7}{2}}x^{-s}.dx+\\
&\int _{N+\frac{5}{2}}^{N+\frac{7}{2}}\left[(N+3)^{-s}-x^{-s}\right].dx+\int _{N+\frac{7}{2}}^{N+\frac{9}{2}}x^{-s}.dx+\\
&\int _{N+\frac{7}{2}}^{N+\frac{9}{2}}\left[(N+4)^{-s}-x^{-s}\right].dx+\ldots \\
&=\underset{n=1}{\overset{n=N}{\sum }}n^{-s}+\int _{N+\frac{1}{2}}^{\infty }x^{-s}.dx+\underset{m=1}{\overset{\infty }{\sum }}\int _{N+\frac{2.m-1}{2}}^{N+\frac{2.m+1}{2}}\left[(N+m)^{-s}-x^{-s}\right].dx
\end{split} \label{A6} \tag{A6}
\end{equation}

From the TAYLOR-McLAURIN expansion of $x^{-s}$   around the point   $x=N+m$,   $1\leq m\leq \infty$: 
\begin{equation}
\begin{split}
&\int _{N+\frac{2.m-1}{2}}^{N+\frac{2.m+1}{2}}\left[(N+m)^{-s}-x^{-s}\right].dx\\
&=\int _{\eta =-\frac{1}{2}}^{\eta =+\frac{1}{2}}\left[(N+m)^{-s}-(N+m+\eta )^{-s}\right].d\eta\\
\\
&(N+m+\eta )^{-s}=\\
&(N+m)^{-s}-s.\eta .(N+m)^{-1-s}+\\
&\frac{s.(1+s)}{2!}.(N+m)^{-2-s}.\eta ^2-\\
&\frac{s.(1+s).(2+s)}{3!}.(N+m)^{-3-s}.\eta ^3+\ldots
\end{split} \label{A7} \tag{A7}
\end{equation}
it is noted that the integrals of the terms with odd exponents of   $\eta$  are  equal to zero, and therefore  only the integrals of the terms with even exponents of   $\eta$ are left:

\begin{equation}
\begin{split}
&\int _{\eta =-\frac{1}{2}}^{\eta =+\frac{1}{2}}\left[(N+m)^{-s}-(N+m+\eta )^{-s}\right].d\eta\\
&=-\int _{\eta =-\frac{1}{2}}^{\eta =+\frac{1}{2}}\left[\frac{s.(1+s)}{2!}.(N+m)^{-2-s}.\eta ^2+\frac{s.(1+s).(2+s).(3+s)}{4!}.(N+m)^{-4-s}.\eta ^4+\ldots\right].d\eta\\
&=-\left[\frac{s.(1+s)}{4.3!}.(N+m)^{-2-s}+\frac{s.(1+s).(2+s).(3+s)}{16.5!}.(N+m)^{-4-s}+\ldots \right]\\
&=-\frac{(N+m)^{-s}}{4}.\underset{k=0}{\overset{\infty }{\sum }}\frac{(N+m)^{-2.(k+1)}}{\Gamma (s)}.\frac{\Gamma (s+2.k+1)}{2^{2.k}.(2.k+3)!}
\end{split} \label{A8} \tag{A8}
\end{equation}

This result should be summed from   $m=1$    to   $m\rightarrow \infty $, by an iterative procedure replacing at each step the first sum (containing powers  $(N+m)^{-2-s},(N+m)^{-4-s}\ldots$ in the successive steps)  with an integral plus  corrective terms   as shown above for the  term in   $n^{-s}$. As a final result,  the structure of the main term and of the corrective terms is derived, as anticipated in footnote 2. The result is the formula  hereunder repeated:
\begin{equation}
Z(s)=\left[\frac{x_0}{s-1}+\underset{k=0}{\overset{\infty }{\sum }}\beta _{2.k+1}.\frac{\Gamma(s+2.k+1)}{\Gamma(s)}-\sum _{n=1}^{n=N}\left(\frac{n}{x_0}\right)^{-s}\right].\exp\left(-s.\ln{x_0}\right) \label{A9} \tag{A9}
\end{equation}

It remains to be discussed how to evaluate the integral  $\int _{N+\frac{1}{2}}^{\infty }x^{-s}.dx$   in \eqref{A6}, which could turn up to be an improper integral on the account of the infinite upper limit whereby, e.g., the integral function $\int x^{-s}.dx$ may diverge to infinity.

Let us begin by recalling the formal structure of the analytical function we are working with:

\begin{equation}
Z(s)=\underset{n=1}{\overset{\infty }{\sum }}n^{-s}=\underset{n=1}{\overset{\infty }{\sum }}e^{-s.\ln n} \label{A10} \tag{A10}
\end{equation}
where:
\begin{equation}
\begin{split}
s=\frac{1}{2}+\xi +i.Y_0\\
-\frac{1}{2}<\xi <+\frac{1}{2}
\end{split}
\label{A11} \tag{A11}
\end{equation}

Under our assumptions Eq.\eqref{A4} gives,   with  $x_0=N+\frac{1}{2}$:

\begin{equation}
\begin{split}
Z(s)=&\underset{n=1}{\overset{N}{\sum }}n^{-s}+\underset{X\rightarrow \infty }{\lim }\frac{X^{1-s}-\left(x_0\right)^{1-s}}{1-s}+\\
&\underset{K,N\rightarrow \infty }{\lim }\left[\frac{\left(x_0\right)^{-s-1}}{\Gamma (s)}.\underset{k=0}{\overset{k=K}{\sum }}\alpha _{2.k+1}.\Gamma (s+2.k+1).\left(x_0\right)^{-2.k}\right]
\end{split}
\label{A12} \tag{A12}
\end{equation}
which is not very helpful. But let us then  imagine to divide the infinite interval of integration into a succession of infinite adjacent intervals of exponentially increasing length:

\begin{equation}
\begin{split}
&x_0.\left(e^{\frac{\pi }{Y_0}}-1\right)+x_0.\left(e^{\frac{2.\pi }{Y_0}}-e^{\frac{\pi }{Y_0}}\right)+x_0.\left(e^{\frac{3.\pi }{Y_0}}-e^{2.\frac{\pi }{Y_0}}\right)+\\
&\ldots +x_0.\left(e^{\frac{j.\pi }{Y_0}}-e^{\frac{(j-1).\pi }{Y_0}}\right)+\ldots
\end{split}
\label{A13} \tag{A13}
\end{equation}

It comes successively:	
\begin{equation}
\begin{split}
Z(s)-\underset{n=1}{\overset{n=N}{\sum }}n^{-s}=&-\underset{j=1}{\overset{\infty }{\sum }}\frac{x_0^{1-s}}{1-s}\left[e^{(1-s).\frac{j.\pi }{Y_0}}+e^{(1-s).\frac{(j-1).\pi }{Y_0}}\right]-\\
&\frac{x_0^{-1-s}}{\Gamma (s)}.\underset{j=1}{\overset{\infty }{\sum }}\underset{k=0}{\overset{\infty }{\sum }}x_0^{-2.k}.\alpha _{2.k+1}.\Gamma (s+2.k+1).\\
&\left[e^{-(2.k+1-s).\frac{j.\pi }{Y_0}}+e^{-(2.k+1-s).\frac{(j-1).\pi }{Y_0}}\right]
\end{split}
\label{A14} \tag{A14}
\end{equation}

Let us  observe that the first sum with respect to index $j$  is the sum of a geometric series with alternating signs (because of the factors  $e^{-j.i.Y_0.\frac{\pi }{Y_0}}=e^{-j.i\pi}=(-1)^j$ ), so that this sum can be obtained by the usual formula:   $\underset{j=0}{\overset{j=\infty }{\sum }}a .r^j=\frac{a}{1-r}$,   despite the fact that   $|r|=\left|e^{-\left(2.k+1-\frac{1}{2}-\xi \right).\frac{j.\pi }{Y_0}}\right|>1$. A similar development, substituting in place of  $n^{-s}$ successively  $n^{-2-s}  ,n^{-4-s}\ldots$, can be used for the iterative procedure leading to the final expression of the corrective terms. Therefore, posing  $\beta _{2.k+1}=\alpha _{2.k+1}.x_0^{-(2.k+1)}$, where  $\alpha _1=\frac{1}{24}$ and
\begin{equation}
\begin{split}
\alpha _{2.k+1}=&-\frac{\alpha _{2.k-1}}{2^2.3!}-\frac{\alpha _{2.k-3}}{2^4.5!} - \ldots .+\frac{1}{2^{2.k+2}.(2.k+3)!}\\
\text{(where }& \lim \left(\frac{\beta _{2.k+1}}{\beta _{2.k-1}}\right)\rightarrow -\frac{1}{4 \pi ^2  x_0^2}   \text{ for }   k\gg 1  \text{),  Eq.\eqref{A12} becomes:}\\
Z(s)-\underset{n=1}{\overset{N}{\sum }}n^{-s}=&-\frac{x_0^{1-s}}{1-s}.\frac{e^{(1-s).\frac{\pi }{Y_0}}+1}{1+e^{(1-s).\frac{\pi }{Y_0}}}-\\
&\frac{x_0^{-s}}{\Gamma (s)}.\underset{k=0}{\overset{\infty }{\sum }}\beta _{2.k+1}.\Gamma (s+2.k+1).\frac{e^{-(2.k+1-s).\frac{\pi }{Y_0}}+1}{1+e^{-(2.k+1-s).\frac{\pi }{Y_0}}}\\
Z(s)-\underset{n=1}{\overset{N}{\sum }}n^{-s}=&-\frac{x_0^{1-s}}{1-s}-\frac{x_0^{-s}}{\Gamma (s)}.\underset{K\rightarrow \infty }{\lim }\underset{k=0}{\overset{K}{\sum }}\beta _{2.k+1}.\Gamma (s+2.k+1)
\end{split}
\label{A15} \tag{A15}
\end{equation}

 It is seen that the convergence of this expression is strongly conditioned by the choice of   $x_0=N+\frac{1}{2}$   and of   $K$.

The equation that makes   $Z(s)$   defined in the critical strip is therefore:
\begin{equation}
Z(s)=\underset{n=1}{\overset{n=N}{\sum }}n^{-s}-\frac{x_0^{1-s}}{1-s}-\frac{x_0^{-s}}{\Gamma (s)}.\underset{K\rightarrow \infty }{\lim }\underset{k=0}{\overset{K}{\sum }}\beta _{2.k+1}.\Gamma (s+2.k+1)
\label{A16} \tag{A16}
\end{equation}
and the equation that defines the zero-points on the C.L. takes on the form:
\begin{equation}
\underset{n=1}{\overset{n=N}{\sum }}\frac{n^{-s}}{x_0^{1-s}}-\frac{1}{1-s}-\frac{1}{\Gamma (s)}.\underset{K\rightarrow \infty }{lim}\underset{k=0}{\overset{K}{\sum }}\beta _{2.k+1}.\Gamma (s+2.k+1)=0
\label{A17} \tag{A17}
\end{equation}
or:
\begin{equation}
\frac{1}{4}+Y_0^2=s.(1-s)=\frac{-s.\Gamma (s)}{\frac{\underset{k=0}{\overset{K}{\sum }}\beta _{2.k+1}.\Gamma (s+2.k+1)}{x_0}-\Gamma (s).\underset{n=1}{\overset{n=N}{\sum }}\frac{n^{-s}}{x_0^{1-s}}} \text{, } s=\frac{1}{2} \pm iY_0
\label{A18} \tag{A18}
\end{equation}

Eq.\eqref{A18} is an implicit equation where the only unknown is  $Y_0$.

%
%
\section{Figures}\label{sec:7}
\begin{figure}[!htb]
\centering
\includegraphics[scale=0.7]{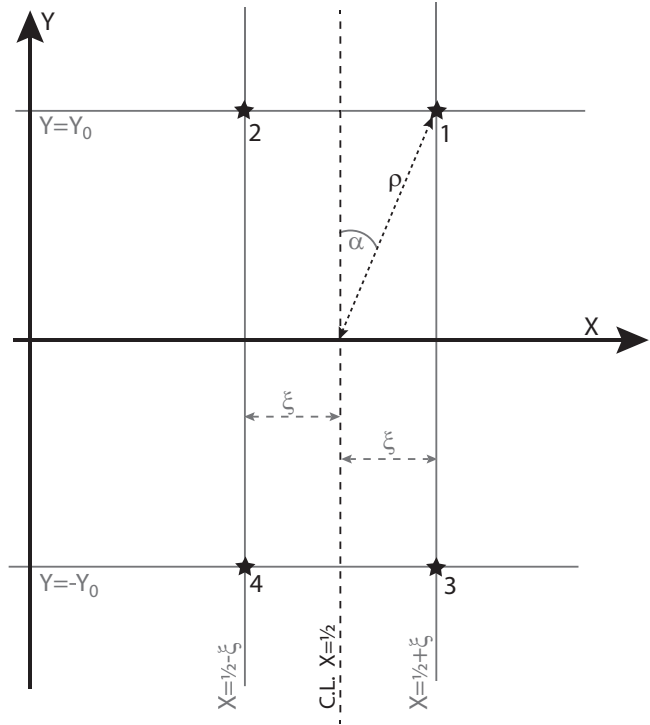}
\caption{Numeration of the four hypothetical outliers  1, 2, 3, 4}
\label{fig:1}
\end{figure}

%
%

\end{document}